%% file: Tmain.tex
\documentclass[abstract=true, DIV=12]{scrartcl}

\usepackage[section]{preamble}
\newcommand{\IIBGtransfer}{Corollary~6.1.2}

\title{The Becker--Gottlieb transfer is not functorial: An example using $\Cinf$-algebras}
\author{Florian Naef \and Robin Stoll}
\date{\today}

\begin{document}

\maketitle

\begin{abstract}
  We prove that the Becker--Gottlieb transfer is not functorial, even on rational cohomology.
  More precisely, for any even $n \ge 4$, we construct a fibration $f \colon X \to \Sphere {n+1}$ with fiber $\Sphere 2 \vee \Sphere n$, and an oriented $\Sphere 1$-bundle $g \colon Y \to X$ such that the transfer of $f \after g$ is non-trivial on rational cohomology.
  Since the transfer of an oriented $\Sphere 1$-bundle vanishes, this implies that the transfer of $f \after g$ and the composite of the transfers of $f$ and $g$ induce different maps on rational cohomology and are thus not homotopic.

  The proof that the transfer of $f \after g$ is non-trivial uses the rational models for the Becker--Gottlieb and THH transfers obtained in previous work of the authors.
  In the appendix, we provide a second, more involved proof that relies on constructing explicit $\Cinf$-algebra models for odd spherical fibrations (including circle bundles), which might be of independent interest.
  
  Concurrently, another counterexample (of a very different flavor) was found by Georg Lehner, Cary Malkiewich, and the first author.
\end{abstract}

\setcounter{tocdepth}{\sectiontocdepth}
\tableofcontents

\numberwithin{equation}{section}

\section{Introduction}

The \emph{Becker--Gottlieb transfer} of a fibration $f \colon X \to Y$ whose fibers are finitely dominated (i.e.\ they are homotopy retracts of finite CW-complexes) is a ``wrong-way'' map
\[ f^! \colon \SS[Y] \longto \SS[X] \]
of their suspension spectra.
It was originally defined by Becker--Gottlieb \cite{BG76} for fibrations whose base is finite-dimensional, and extended by Clapp \cite{Cla} to the general case (also see \cite[§6]{CCRY} for a modern treatment).
It is a natural question whether this transfer is \emph{functorial}, i.e.\ whether $(f \after g)^! \eq g^! \after f^!$ whenever $f$ and $g$ have finitely dominated fibers.\footnote{This implies that the fibers of $f \after g$ are also finitely dominated, see e.g.\ \cite[Proposition~C.6]{Bar17}.}
This holds when $f$ is a fiber bundle with fiber a compact manifold, by work of Dwyer--Weiss--Williams \cite{DWW} and Lind--Malkiewich \cite{LM}, and it always holds on the level of $\pi_0$, by work of Klein--Malkiewich--Ramzi \cite{KMR} (also see \cite[§6.3]{CCRY} for a summary of what is known).
Surprisingly, the general statement has remained open for the past 50 years; there have been multiple attempts at proving it \cite{Hau13,KM}, but each of them turned out to contain an unfixable mistake.

In this paper, we prove that the statement is in fact false by providing a counterexample.
Moreover, the failure of functoriality is already detected on rational cohomology.
Simultaneously to our work, Lehner--Malkiewich--Naef \cite{LMN} also obtained a counterexample.
It is, however, of a very different flavor from ours; in particular, they detect the failure of functoriality in the torsion of stable homotopy groups.

\begin{introtheorem}[see \cref{thm:main}] \label{intro:thm:main}
  The Becker--Gottlieb transfer is not functorial, even on rational cohomology.
  That is, there exist two maps of spaces with finitely dominated fibers
  \[ X  \xlongto{g}  Y  \xlongto{f}  Z \]
  such that the maps of spectra
  \[ (f \after g)^! \colon \SS[Z]  \longto  \SS[X]  \displayand  g^! \after f^! \colon \SS[Z]  \longto  \SS[X] \]
  induce different maps on rational cohomology.
  In particular $(f \after g)^!$ and $g^! \after f^!$ are not homotopic.
\end{introtheorem}

Our counterexample is completely explicit: for any even $n \ge 4$, we consider two fibrations
\begin{equation} \label{intro:eq:example}
  E_\xi  \xlongto{g}  E  \xlongto{f}  \Sphere {n+1}
\end{equation}
where the fibers of $f$ are $\Sphere 2 \vee \Sphere n$, and $g$ is the oriented $\Sphere 1$-bundle classified by any generator $\xi \in \Coho 2 (E; \ZZ) \iso \ZZ$.\footnote{The example also works for $n = 2$, but in that case one needs to be more careful in specifying $\xi$.}
Writing $\B \aut_*(\Sphere 2 \vee \Sphere n)$ for the classifying space of the topological monoid of pointed self--homotopy equivalences of $\Sphere 2 \vee \Sphere n$, the fibration $f$ is classified by the map $\Sphere {n+1} \to \B \aut_*(\Sphere 2 \vee \Sphere n)$ corresponding to the element of
\begin{equation} \label{intro:eq:aut_wedge}
  \begin{aligned}
  \pi_n \bigl( \aut_*(\Sphere 2 \vee \Sphere n), \id \bigr)  &\iso  \pi_n \bigl( \Loops^2 (\Sphere 2 \vee \Sphere n) \times \Loops^n (\Sphere 2 \vee \Sphere n), (\inc, \inc) \bigr) \\
  &\iso  \pi_{n + 2} ( \Sphere 2 \vee \Sphere n ) \oplus \pi_{2n} ( \Sphere 2 \vee \Sphere n )
  \end{aligned}
\end{equation}
given by $(0, [\iota_n, [\iota_n, \iota_2]])$, where $\iota_i \colon \Sphere i \to \Sphere 2 \vee \Sphere n$ is the respective inclusion, and $[\blank, \blank]$ denotes the Whitehead product.
Note that $(f \after g)^! \not\eq g^! \after f^!$ implies that $f$ is not fiberwise homotopy equivalent to a fiber bundle over $\Sphere {n+1}$ whose fiber is a compact manifold (potentially with boundary).

\subsection*{The proof}

It is a well-known fact that the Becker--Gottlieb transfer of any oriented $\Sphere 1$-bundle vanishes (more generally, it vanishes for any smooth fiber bundle whose vertical tangent bundle admits a nowhere vanishing section).
Hence, to prove \cref{intro:thm:main}, it is enough to show that $(f \after g)^!$ is non-trivial on rational cohomology.

To this end, given a fibration $f \colon E \to B$ with finitely dominated fibers, and an oriented $\Sphere 1$-bundle $g \colon E_\xi \to E$ classified by some $\xi \in \Coho 2(E; \ZZ)$, we provide a description of the transfer $(f \after g)^!$ on rational cohomology in terms of $\xi$ and a rational model of $f$.
Before we state it, recall that a $\Cinf$-algebra $S$ is a cochain complex equipped with a $\k$-bilinear multiplication $\mu^S_2$ that is commutative, and homotopy associative up to specified higher coherences $\mu^S_m$.
Their homotopy theory is equivalent to the homotopy theory of commutative differential graded algebras (cdgas), and hence they can be used as rational models for nilpotent spaces (as in the work of Sullivan \cite{Sul}); under this correspondence, the higher operations $\mu^S_m$ can be thought of as Massey products.
We then prove (see \cref{prop:transfer_circle_composite}) that, on rational cohomology, the transfer $(f \after g)^!$ factors as the Gysin homomorphism $\Coho * (E_\xi; \QQ) \to \Coho {* - 1} (E; \QQ)$ followed by the map induced on cohomology by the map of cochain complexes
\begin{equation} \label{intro:eq:transfer_composite}
  S  \longto  \k,  \qquad  x  \longmapsto  - \tr_\k \bigl( \mu^S_3(x, \blank, \xi) \bigr)
\end{equation}
where $\k$ is a cdga that models $B$, and $S$ is a dualizable $\Cinf$-algebra over $\k$ that models $E$ over $B$; here $\tr_\k$ denotes the $\k$-linear trace of an endomorphism of $S$.

To prove this claim, we use that, by work of Lind--Malkiewich \cite{LM}, the Becker--Gottlieb transfer can be expressed in terms of the transfer in topological Hochschild homology; by our previous work \cite{I}, this THH transfer is rationally modeled by the transfer in ordinary Hochschild homology of cdgas (e.g.\ as defined by Keller \cite{Kel21}).
This transfer \emph{is} functorial, and we evaluate it separately for $f$ and $g$ using the explicit description we obtained in the follow-up \cite{II}.

In view of \eqref{intro:eq:transfer_composite}, to prove \cref{intro:thm:main} it is enough to construct a fibration $f \colon X \to B$ with sufficiently non-trivial ``fiberwise Massey product'' $\mu^S_3$.
We do this using a result contained in forthcoming work of Berglund \cite{Ber} (restated as \cref{prop:Berglund}), who constructs an explicit $\Cinf$-algebra model for a fibration given a dg Lie model of its holonomy action.
In the case where the base of $f$ is a sphere and the fibers are a wedge of spheres, it roughly states that $\mu^S_m$ is determined by $(m-1)$-fold iterated Whitehead products under the isomorphism \eqref{intro:eq:aut_wedge}.
The map $f$ in our example \eqref{intro:eq:example} is one of the simplest ways of realizing a non-trivial $\mu^S_3$.

\subsection*{An alternative proof using $\Cinf$-algebra models for odd spherical fibrations}

In the appendix, we provide an alternative, more conceptual (but also more involved) proof of the fact that, for $g \colon E_\xi \to E$ an oriented $\Sphere 1$-bundle, the transfer $(f \after g)^!$ on rational cohomology is given in terms of the formula \eqref{intro:eq:transfer_composite}.
This relies directly on the rational description of the Becker--Gottlieb transfer obtained in our previous work \cite{II} (via the result of Lind--Malkiewich \cite{LM} and our model for the transfer in topological Hochschild homology).
More precisely, we proved there that the Becker--Gottlieb transfer of a fibration $f$ is, on rational cohomology, induced by the map
\begin{equation*}
  S  \longto  \k,  \qquad  x  \longmapsto  \tr_\k \bigl( \mu^S_2(x, \blank) \bigr)
\end{equation*}
where $S$ is a $\Cinf$-algebra over a cdga $\k$ that together model $f$.
Together with the following result, which provides a $\Cinf$-algebra model for $E_\xi$ in terms of $S$ and $\xi$, this implies the claim.
(Note that $\Sphere n$ is rationally equivalent to $\K(\QQ, n)$ when $n$ is odd.)

\begin{introtheorem}[see \cref{thm:Cinf_circle}]
  Let $d \ge 2$ be an even integer, $f \colon E \to B$ a fibration of nilpotent spaces of finite rational type, $\xi$ an element of $\Coho d (E; \QQ)$, and $g \colon E_\xi \to E$ the principal $\K(\QQ, d-1)$-fibration classified by $\xi$.
  Furthermore let $\k$ be a cofibrant cdga that models $B$, and let $S$ be a unital $\Cinf$-algebra over $\k$ that models $E$ over $B$.
  Then there is an explicit $\Cinf$-algebra $S_\xi$ over $\k$ that models $E_\xi$ over $B$; its structure map $\mu_n$ is defined in terms of the structure maps $\mu^S_n, \dots, \mu^S_{2n}$ of $S$ and any cocycle representative of the class $\xi$.
\end{introtheorem}

When $\xi = 0$, this theorem essentially recovers a special case of a result of Lupton--Smith \cite[Theorem~2.1]{LS07}.

\subsection*{Acknowledgments}

We would like to thank Georg Lehner and Cary Malkiewich for informing us about their simultaneous work on \cite{LMN}, and agreeing to release it at the same time as the present paper.

\subsection*{AI declaration}

While finishing \cite{II}, the authors suspected that the rational description for the Becker--Gottlieb transfer obtained there could be used to construct a counterexample to it being functorial.
However, our approach got stuck at some point, and in retrospect turned out to be needlessly complicated (in essence, we were trying to compute composites of Becker--Gottlieb transfers in certain universal examples using graph complexes).
Then, in September 2026, the first author prompted an LLM to produce a counterexample, upon which it suggested \eqref{intro:eq:example} for $n = 2$; its proof was a direct computation using cdga models, based on \cite{II}.
In this paper, we independently construct a counterexample, starting from the idea to consider $\Sphere 1$-bundles, without further AI input.
Our example ended up being the same as the one suggested by the LLM (up the changing the parameter $n$); this makes sense, considering we expect it to be one of the simplest possible ones.
However, we believe our proofs to be significantly more conceptual.

\subsection*{Notation and conventions}

We will generally work in the $\infty$-category of spaces (though this does not matter much for the purposes of this paper), and call its objects \emph{animas}.

We use homological grading conventions for dg Lie algebras (i.e.\ differentials have degree~$-1$), and write homological degrees as subscripts.
For everything else, we use cohomological grading conventions (i.e.\ differentials have degree~$+1$), and we write cohomological degrees as superscripts.
We denote by $\shift (\blank)$ a homological shift by $+1$, and a cohomological shift by $-1$.

We assume knowledge of the basics of rational homotopy theory, roughly on the level of Félix--Halperin--Thomas \cite[Chs.~II and IV]{FHT}.
For cdga models, we follow the model categorical approach of Bousfield--Gugenheim \cite{BG}; also see \cite[§§2.2--2.4]{I} for a short summary.
In particular, a cofibration of cdgas is a map of non-negatively graded cdgas that is a cofibration in the Bousfield--Gugenheim model structure.

We will furthermore require some familiarity with $\Cinf$-algebras, for which we follow the conventions of \cite[§3]{II} (which are taken from \cite{Kel}).
In particular, for us a $\Cinf$-algebra $S$ over a cdga $\k$ is a (dg) $\k$-module $S$ equipped with structure maps $\mu^S_n \colon (\shift S)^{\tensor_\k n} \to \shift S$; see \cref{app:Cinf_circle} for a short summary of the precise definitions.

\numberwithin{equation}{theorem}

\section{The transfer of composites with circle bundles}

As explained in the introduction, the core of our argument consists of a rational description of the transfer $(f \after g)^!$ when $g$ is an $\Sphere 1$-bundle.
This is provided by the following proposition, which we deduce from our previous work \cite{I,II} on the THH transfer, and its relation to the Becker--Gottlieb transfer due to Lind--Malkiewich \cite{LM}.
In \cref{prop:transfer_circle_composite_alt}, we provide an alternative proof directly based on our rational model for the Becker--Gottlieb transfer from \cite{II}, combined with the explicit $\Cinf$-model for circle bundles that we construct in \cref{app:Cinf_circle}.

Before we state the proposition, recall that a unital $\Cinf$-algebra $S$ over a cdga $\k$ being a \emph{model} for a map $f \colon X \to B$ of nilpotent animas of finite rational type means that there is some chosen cofibration of cdgas $\k \to R$ that models $f$ and a chosen unital quasi-isomorphism $S \to R$ of $\Cinf$-algebras.
In particular there are isomorphisms $\Coho * (S) \iso \Coho * (R) \iso \Coho * (X; \QQ)$ of graded algebras.

\begin{proposition} \label{prop:transfer_circle_composite}
  Let $f \colon X \to B$ be a map of simply connected animas of finite rational type with simply connected compact fiber $F$.
  Furthermore, let $\xi \in \Coho 2 (X; \ZZ)$ such that its restriction along the inclusion $\iota \colon F \to X$ is non-divisible (i.e.\ there exists no $\zeta \in \Coho 2 (F; \ZZ)$ such that $k \zeta = \iota^*(\xi)$ for some integer $k \ge 2$), and let $g \colon X_\xi \to X$ be the oriented $\Sphere 1$-bundle classified by $\xi$.
  Then the Becker--Gottlieb transfer $(f \after g)^! \colon \Coho * (X_\xi; \QQ) \to \Coho * (B; \QQ)$ factors as
  \[ \Coho * (X_\xi; \QQ)  \longto  \Coho {* - 1}(X; \QQ)  \longto  \Coho {*} (B; \QQ) \]
  where the first map is the Gysin homomorphism and the second map is induced by the map of cochain complexes
  \[ \shift S  \longto  \k,  \qquad  \shift x  \longmapsto  (-1)^{\deg {\shift x}} \tr^{\shift S}_\k \bigl( \mu^S_3(\shift x, \blank, \shift \alpha) \bigr) \]
  where $\k$ is a cofibrant cdga model for $B$ and $S$ is a $\Cinf$-algebra over $\k$ that models $X$ over $B$, such that $S$ is cofibrant and dualizable as a $\k$-module, and $\alpha \in S^2$ is a cocyle that represents the image of $\xi$ in $\Coho 2 (X ; \QQ)$.
\end{proposition}

\begin{proof}
  By assumption, there exists a cofibration of cdgas $\k \to R$ that models $f$, and a unital quasi-isomorphism of $\Cinf$-algebras $\phi \colon S \to R$.
  By an abuse of notation, we will also denote the image $\phi_1(\alpha)$ by $\alpha$.
  Then $g \colon X_\xi \to X$ is modeled by the cofibration $R \to R_\alpha$, where $R_\alpha \defeq R \tensor \SA(a)$ with $\deg a = 1$ and differential determined by $d(a) \defeq \alpha \in R$ (see e.g.\ \cite[§15(b), Example~4]{FHT}).

  We claim that $\fib(f \after g)$ is simply connected.
  To this end, note that there is an oriented $\Sphere 1$-bundle $f' \colon \fib(f \after g) \to F$, which is classified by $\iota^*(\xi) \in \Coho 2 (F; \ZZ)$.
  Since $F$ is simply connected, the fundamental group of $\fib(f \after g)$ is a quotient of $\ZZ$ and hence abelian.
  In particular it is enough to prove that $\Ho 1 (\fib(f \after g))$ vanishes.
  In the Serre spectral sequence for $f'$, the only contribution to this homology group is from $\Ho 0(F; \Ho 1(\Sphere 1)) \iso \ZZ$, which is hit by the differential $d_2 \colon \Ho 2(F; \ZZ) \to \ZZ$ given by the image of $\iota^*(\xi)$ under $\ev \colon \Coho 2(F; \ZZ) \to \Hom(\Ho 2(F; \ZZ), \ZZ)$.
  Since $F$ is simply connected, the map $\ev$ is an isomorphism by the universal coefficient theorem, and hence $\ev(\iota^*(\xi))$ is non-divisible.
  If $\ev(\iota^*(\xi))$ were not surjective, then its image would be contained in a proper subgroup $(p) \subset \ZZ$ for some prime $p$; hence $\ev(\iota^*(\xi))$ would be divisible by $p$, which is a contradiction.
  Thus $\ev(\iota^*(\xi))$ is surjective, and so $\fib(f \after g)$ is simply connected.
  
  By \cite[Theorem~1.2]{LM} (see also \cite[Proposition~6.1.1]{II}), the Becker--Gottlieb transfer $(f \after g)^!$ is homotopic to the composite
  \[ \SS[B]  \xlongto{\alpha}  \THH(B)  \xlongto{(f \after g)^*}  \THH(E_\xi)  \xlongto{e}  \SS[E_\xi] \]
  where $\alpha$ is the THH assembly map, the middle map is the THH transfer of $f \after g$, and $e$ is given by the equivalence $\THH(E_\xi) \eq \SS[\Map(\Sphere 1, E_\xi)]$ and evaluation at the basepoint of $\Sphere 1$.
  By \cite[Theorem~4.6.3 and Lemmas~4.5.4 and 4.5.5]{I} (and using that the fiber of $f \after g$ is simply connected), this composite is modeled by the composite
  \[ R_\alpha  \longto  \HH_\k(R_\alpha)  \xlongto{\transfer{R_\alpha}}  \HH_\k(\k)  \xlongto{\eq}  \k \]
  where $\HH_\k$ denotes the Hochschild complex over $\k$ (see e.g.\ \cite[Definition~2.2.13]{I}), the first map is given by the inclusion, the middle map is the Hochschild homology transfer associated to $\k \to R_\alpha$ (see e.g.\ \cite[Construction~4.6.1]{I}), and the final map is the projection onto Hochschild degree $0$.
  The middle map factors as
  \[ \HH_\k(R_\alpha)  \xlongto{\transfer{R_\alpha}}  \HH_\k(R)  \xlongto{\transfer{R}}  \HH_\k(\k) \]
  since the Hochschild homology transfer is functorial (see Armenta--Keller \cite[Theorem~2.1]{AK}).

  Since $R_\alpha$ is dualizable over $R$, the transfer $\transfer{R_\alpha} \colon \HH_\k(R_\alpha) \to \HH_\k(R)$ can be explicitly described using \cite[Theorem~4.2.5]{II}.
  Using the ``derived coevaluation''
  \[ c  \defeq  1 \otimes \dual 1 + a \otimes \dual a + 1 \otimes \shift \alpha \otimes \dual a  \in  R_\alpha \otimes_R \B_\k(R) \otimes_R \dual{R_\alpha} \]
  this result implies that the composite $R_\alpha \to \HH_\k(R_\alpha) \to \HH_\k(R)$ is (homotopic to the map) given by
  \begin{align*}
      &x  \longmapsto  0  \displayand  x \cdot a  \longmapsto  \shift[-1] ( \shift x \tensor \shift \alpha )
      \intertext{for $x \in R$. In particular, it factors through the map $R_\alpha \to \shift R$ given by}
      &x  \longmapsto  0  \displayand  x \cdot a  \longmapsto  (-1)^{\deg x} \shift x
  \end{align*}
  which is a model for the Gysin homomorphism (this follows from the identification of the Serre spectral sequence for $X_\xi \to X$ in terms of $R \to R_\alpha$ as provided in \cite[End of proof of Proposition~3.2]{Pri}).
  We furthermore note that the following square commutes
  \[
  \begin{tikzcd}
    \shift S \rar \dar[swap]{\phi_1} & \HH_\k(S) \dar{\phi_*} \\
    \shift R \rar & \HH_\k(R)
  \end{tikzcd}
  \]
  where $\HH_\k(S)$ denotes the Hochschild complex of the $\Cinf$-algebra $S$ over $\k$ (see e.g.\ \cite[§3.7]{II}), and both horizontal maps are given by $\shift x \mapsto (-1)^{\deg x} \shift[-1] (\shift x \tensor \shift \alpha)$.
  To see this note that (e.g.\ by \cite[Observation~3.7.6]{II})
  \[ \phi_*(\shift x \tensor \shift \alpha)  =  \phi_1(\shift x) \tensor \phi_1(\shift \alpha) + \phi_2(\shift x, \shift \alpha) + (-1)^{\deg {\shift x}} \phi_2(\shift \alpha, \shift x) = \phi_1(\shift x) \tensor \phi_1(\shift \alpha) \]
  using the shuffle relation for the $\Cinf$-morphism $\phi$.
  Note that $\phi_*$ is a quasi-isomorphism by \cite[Lemma~3.7.8]{II}.

  We now consider the diagram
  \[
  \begin{tikzcd}
     & \shift S \rar \dar{\eq}[swap]{\phi_1} & \HH_\k(S) \rar{\transfer{S}} \dar{\phi_*}[swap]{\eq} & \HH_\k(\k) \rar{\eq} \dar[equal] & \k \dar[equal] \\
    R_\alpha \rar & \shift R \rar & \HH_\k(R) \rar{\transfer{R}} & \HH_\k(\k) \rar{\eq} & \k
  \end{tikzcd}
  \]
  where the middle square commutes up to homotopy by \cite[Proposition~4.1.4]{II}.
  By \cite[Theorem~4.2.7]{II}, the composite $\HH_\k(S) \to \HH_\k(\k) \to \k$ is given by
  \[ \shift[-1] ( \shift x \tensor \shift \alpha)  \longmapsto  \tr_\k^{\shift S} \bigl( \mu^S_3 ( \shift x, \shift \alpha, \blank) \bigr) + (-1)^{\deg {\shift x}} \tr_\k^{\shift S} \bigl( \mu^S_3 ( \shift \alpha, \shift x, \blank) \bigr)  =  - \tr_\k^{\shift S} \bigl( \mu^S_3 ( \shift x, \blank, \shift \alpha) \bigr) \]
  using the shuffle relation for $\mu^S_3$ (regarding the sign, note that we have $\mu^S_3(\shift x, \blank, \shift \alpha)(\shift y) = (-1)^{\deg {\shift y}} \mu^S_3(\shift x, \shift y, \shift \alpha)$).
  This completes the proof.
\end{proof}

\section{The counterexample}

In this section, we construct our counterexample to functoriality of the Becker--Gottlieb transfer.
As explained in the introduction, it rests on the following well-known fact.
(The assumption that $X$ is compact is not necessary, but we did not find a convenient reference for the more general case.)

\begin{proposition}[Becker--Gottlieb] \label{prop:BG_circle}
  Let $g \colon Y \to X$ be a smooth fiber bundle with fiber a compact manifold, and assume that $X$ is a compact anima.
  If the vertical tangent bundle of $g$ over $Y$ admits a nowhere vanishing vector field, then the Becker--Gottlieb transfer $g^! \colon \SS[X] \to \SS[Y]$ is trivial.
\end{proposition}

\begin{proof}
  First suppose that $X$ is a finite complex.
  Then \cite[Theorem~7.1]{BG91} states that the transfer is given by the global vector field transfer defined in \cite[§5]{BG91}, which vanishes by construction if the vertical tangent bundle has a nowhere vanishing vector field.
  If $X$ is only compact (i.e.\ finitely dominated), then it is a retract of a finite complex; since the Becker--Gottlieb transfer is functorial with respect to base change (see e.g.\ \cite[Lemma~6.19]{CCRY}), this completes the proof.
\end{proof}

In particular this applies to any oriented $\Sphere 1$-bundle (which are classified by $\B \SO(2) \eq \B \Sphere 1 \eq \K(\ZZ, 2)$).
Hence, to contradict functoriality of the Becker--Gottlieb transfer, it is enough to construct a map of animas $f \colon X \to B$ with compact fibers, and an oriented $\Sphere 1$-bundle $g \colon Y \to X$, such that the Becker--Gottlieb transfer $(f \after g)^!$ is non-trivial.
By \cref{prop:transfer_circle_composite} it is thus enough to find (omitting technical conditions):
\begin{itemize}
  \item
  a map of simply connected animas $f \colon X \to B$ with compact simply connected fiber $F$,
  \item
  a cohomology class $\xi \in \Coho 2 (X; \ZZ)$ whose restriction to $F$ is non-trivial,
  \item
  a cdga $\k$ that models $B$,
  \item
  a dualizable $\Cinf$-algebra $S$ over $\k$ that models $X$ over $B$,
  \item
  a cocycle $\alpha \in S^2$ that represents the image of $\xi$ in $\Coho 2 (X; \QQ)$,
  \item
  a cocycle $x \in S^n$ for some $n$,
\end{itemize}
such that $\tr^{\shift S}_\k ( \mu^S_3(\shift x, \blank, \shift \alpha) ) \in \Coho {n+1} (\k)$ is non-trivial.

We will do this using the following result appearing in forthcoming work of Berglund \cite{Ber}.
Before we state it, recall that $\Der(L)$ denotes the dg Lie algebra of derivations of a dg Lie algebra $L$; when $L$ is the dg Lie model of some pointed nilpotent space $X$, the positive homological truncation $\Der(L)_{\ge 1}$ is a dg Lie model for the classifying space $\B \aut_*(X)_\id$ of the identity component of the topological monoid of pointed self--homotopy equivalences of $X$ (see \cite[Corollaire~VII.4.(4)]{Tan}).
We denote by $\CoCE(\lie g; M)$ the Chevalley--Eilenberg cochain complex of a dg Lie algebra $\lie g$ with coefficients in a $\lie g$-module $M$ (see e.g.\ \cite[§23(c)]{FHT}).

\begin{proposition}[Berglund] \label{prop:Berglund}
  Let $B$ be a nilpotent anima, $F$ a compact simply connected anima of finite type, $c \colon B \to \B \aut_*(F)_\id$ a map, and $f \colon E \to B$ the map (with fiber $F$) classified by $c$.
  Write $H_* \defeq \Ho * (F; \QQ)$ and $\wt{H}_* \defeq \rHo * (F; \QQ)$, and similarly for $H^*$ and $\wt{H}^*$.
  Let $L = \LL(\shift[-1] \wt{H}_*)$ be the minimal dg Lie model of $F$, and let $\phi \colon \lie g \to \Der(L)_{\ge 1}$ be a map of dg Lie algebras that models $c$.
  Then $\CoCE(\lie g)$ is a cofibrant cdga model for $B$, and the cofibrant $\CoCE(\lie g)$-module $\CoCE(\lie g; H^*)$ can be equipped with the structure of a unital $\Cinf$-algebra over $\CoCE(\lie g)$ such that it models $E$ over $B$.
  Here $\lie g$ acts on $H^* \iso \QQ \langle 1 \rangle \oplus \wt{H}^*$ by the dual of the linear part of its action on $L$.
  The $\CoCE(\lie g)$-linear $\Cinf$-structure maps $\mu_n$ of $\CoCE(\lie g; H^*)$ are such that the restriction
  \begin{equation*}
    (\shift H^*)^{\tensor n}  \longto  \CoCE \bigl( \lie g; (\shift H^*)^{\tensor n} \bigr)  \iso  \bigl( \shift \CoCE(\lie g; H^*) \bigr)^{\tensor_{\CoCE(\lie g)} n}  \xlongto{\mu_n}  \shift \CoCE(\lie g; H^*)
  \end{equation*}
  is the sum $m_n + \psi_n$, where $m_n \colon (\shift H^*)^{\tensor n} \to \shift H^*$ is the $\Cinf$-algebra structure on $H^*$, and $\psi_n \colon (\shift H^*)^{\tensor n} \to \Hom_\QQ(\lie g, \shift H^*)$ is the adjoint of the map
  \[ x_1 \tensor \dots \tensor x_n \tensor \theta  \longmapsto  \sum_i \langle x_1 \otimes \dots \otimes x_n, \phi(\theta)(\dual{e_i}) \rangle \cdot e_i \]
  where the $e_i$ are a homogeneous basis of $\shift \wt{H}^*$, the $\dual{e_i}$ are the corresponding dual basis of $\dual{(\shift \wt{H}^*)} \iso \shift[-1] \wt{H}_*$, and $\langle \blank, \blank \rangle$ denotes the canonical pairing $\T(\shift H^*) \tensor \T(\shift[-1] H_*) \to \QQ$.
  (In the formula we consider $\dual{e_i} \in \shift[-1] \wt{H}_*$ as an element of $L$; hence $\phi(\theta)(\dual{e_i})$ lies in $L$, which we interpret as a subspace of $\T(\shift[-1] H_*)$.)
\end{proposition}

\begin{proof}
  This will appear in \cite{Ber}; also see \cite[Proposition~5.3.4]{II} for a proof sketch.
\end{proof}

Writing $\k \defeq \CoCE(\lie g)$ and $S \defeq \CoCE(\lie g; H^*)$, this result tells us that for $\tr^{\shift S}_\k ( \mu^S_3(\shift x, \blank, \shift \alpha) )$ to be non-trivial, it is essentially enough to produce a derivation of $L = \LL(\shift[-1] \wt{H}_*)$ that sends some $e \in \wt{H}_*$ to $[\dual x, [e, \dual \alpha]]$.
This is very easy to do, for example by choosing $X$ to be a wedge of spheres, so that the differential of $L$ is trivial.
We first record that, in this case, such derivations are also easy to realize topologically.

\begin{lemma} \label{lemma:aut_wedge_spheres}
  For $1 \le i \le k$, let $n_i \ge 2$ be an integer.
  Writing $X \defeq \bigvee_i \Sphere{n_i}$ and $V \defeq \shift[-1] \rHo * (X; \QQ)$, the following diagram of abelian groups commutes for every $p \ge 1$
  \[
  \begin{tikzcd}
    \pi_p \bigl( \aut_* (X), \id \bigr) \rar{\iso} \dar & \pi_p \bigl( \prod_j \Loops^{n_j} X, (\inc_j)^k_{j = 1} \bigr) \rar{\iso} & \bigoplus_j \pi_{p + n_j} (X) \dar{\rho} \\
    \Der(\LL V)_p \rar{\iso} & \Hom(V, \LL V)_p \rar{\iso}[swap]{\psi} & \bigoplus_j (\LL V)_{p + n_j - 1}
  \end{tikzcd}
  \]
  where the vertical maps are induced by rationalization, using that $\LL V$ (equipped with the trivial differential) models $X$ and that $\Der (\LL V)_{\ge 1}$ models $\B \aut_*(X)_\id$.
\end{lemma}

\begin{proof}
  Let $x_0 \in X$ and $s_0 \in \Sphere p$ be the basepoints.
  Note that (for $p \ge 1$) a pointed map $\phi \colon (\Sphere p, s_0) \to (\aut_*(X), \id)$ is equivalently a map $f \colon (\Sphere p \times X) / (\Sphere p \times x_0) \to X$ such that $f(s_0, x) = x$ for all $x \in X$.
  We observe
  \[ (\Sphere p \times X) / (\Sphere p \times x_0)  \iso  \bigvee_i \bigl( (\Sphere p \times \Sphere {n_i}) / (\Sphere p \times x_0) \bigr)  \eq  \bigvee_i (\Sphere {n_i} \vee \Sphere{p + n_i}) \]
  using, for the last homotopy equivalence, that the attaching map $\Sphere {p + n_i - 1} \to \Sphere {n_i}$ of the top cell of $(\Sphere p \times \Sphere {n_i}) / (\Sphere p \times x_0)$ is null-homotopic via the composite $\Disk {p + n_i} \to \Sphere p \times \Sphere {n_i} \to \Sphere {n_i}$.
  One then checks that an element $\gamma = (\gamma_1, \dots, \gamma_k)$ of $\bigoplus_j \pi_{p + n_j}(X)$ corresponds to the map $f_\gamma \colon \bigvee_i (\Sphere {n_i} \vee \Sphere{p + n_i}) \to X$ given by $\bigvee_i (\inc_i \vee \gamma_i)$.
  Note that $f_\gamma$ is modeled by the map $\LL (V \oplus \shift[p] V) \to \LL (V)$ given by $v \mapsto v$ and $\shift[p] v \mapsto \inv{\psi}(\rho(\gamma))(v)$ for all $v \in V$.
  Then the claim follows from an argument of Lindell--Saleh \cite[Proof of Proposition~3.12]{LS24} (which builds on work of Lupton--Smith \cite[§3]{LS07}). 
\end{proof}

\begin{proposition} \label{prop:example}
  Let $n \ge 4$ be an even integer, and $c \colon \Sphere {n+1} \to \B \aut_*(\Sphere 2 \vee \Sphere n)_\id$ the map corresponding to the element of
  \begin{align*}
    \pi_n \bigl( \aut_*(\Sphere 2 \vee \Sphere n), \id \bigr)  &\iso  \pi_n \bigl( \Loops^2 (\Sphere 2 \vee \Sphere n) \times \Loops^n (\Sphere 2 \vee \Sphere n), (\inc, \inc) \bigr) \\
    &\iso  \pi_{n + 2} ( \Sphere 2 \vee \Sphere n ) \oplus \pi_{2n} ( \Sphere 2 \vee \Sphere n ) 
  \end{align*}
  given by $(0, [\iota_n, [\iota_n, \iota_2]])$, where $\iota_i \colon \Sphere i \to \Sphere 2 \vee \Sphere n$ is the respective inclusion.
  Let $f \colon E \to \Sphere {n+1}$ be the fibration classified by $c$, and $g \colon E_\xi \to E$ the oriented $\Sphere 1$-bundle classified by a generator $\xi \in \Coho 2 (E; \ZZ) \iso \ZZ$.
  Then the Becker--Gottlieb transfer
  \[ \Coho {n+1} (E_\xi; \QQ)  \xlongto{(f \after g)^!}  \Coho {n+1} (\Sphere {n+1}; \QQ) \]
  is non-trivial.
\end{proposition}

\begin{proof}
  Let $z$ be a generator of $\Ho {n+1} (\Sphere {n+1}; \QQ)$, and $y_2$ and $y_n$ the images of generators of $\Ho 2 (\Sphere 2; \QQ)$ and $\Ho n (\Sphere n; \QQ)$ in $\Ho * (\Sphere 2 \vee \Sphere n; \QQ)$.
  The minimal dg Lie models of $\Sphere{n+1}$ and $\Sphere 2 \vee \Sphere n$ are $\lie g \defeq \LL(\shift[-1] z)$ and $L \defeq \LL(\shift[-1] y_2, \shift[-1] y_n)$, respectively equipped with the trivial differential.
  Denoting by $\theta$ the derivation of $L$ of degree $n$ determined by
  \[ \shift[-1] y_n \longmapsto [\shift[-1] y_n, [\shift[-1] y_n, \shift[-1] y_2]]  \displayand  \shift[-1] y_2 \longmapsto 0 \]
  we have that the map
  \[ \phi \colon \LL(\shift[-1] z) \longto \Der(\LL(\shift[-1] y_2, \shift[-1] y_n))_{\ge 1}, \qquad \shift[-1] z \longmapsto \theta \]
  is a dg Lie model for $c$, by \cref{lemma:aut_wedge_spheres}.
  Then $\CoCE(\lie g) \iso \SA(\dual z)$ models $\Sphere {n+1}$ (here we use that $n$ is even), and, by \cref{prop:Berglund}, there is a unital $\Cinf$-algebra structure on $S \defeq \CoCE(\lie g) \tensor \Coho * (\Sphere 2 \vee \Sphere n; \QQ)$ that models $E$ such that
  \begin{equation} \label{eq:example_mu_3}
  \begin{aligned}
    \mu_3 \bigl( \shift (1 \tensor \dual{y_n}), \shift (1 \tensor 1), \shift (1 \tensor \dual{y_2}) \bigr)  &=  0 \\
    \mu_3 \bigl( \shift (1 \tensor \dual{y_n}), \shift (1 \tensor \dual{y_2}), \shift (1 \tensor \dual{y_2}) \bigr)  &=  0 \\
    \mu_3 \bigl( \shift (1 \tensor \dual{y_n}), \shift (1 \tensor \dual{y_n}), \shift (1 \tensor \dual{y_2}) \bigr)  &=  \shift ( \dual z \tensor \dual{y_n} )
  \end{aligned}
  \end{equation}
  (note that $\Sphere 2 \vee \Sphere n$ is formal, so that the $\Cinf$-structure maps of $\Coho * (\Sphere 2 \vee \Sphere n; \QQ)$ are trivial in arity $\neq 2$).

  Note that, by the Serre spectral sequence, the restriction of $\xi \in \Coho 2 (E; \ZZ)$ is non-trivial in the cohomology of the fiber $\Sphere 2 \vee \Sphere n$.
  Moreover, for degree reasons (using that $n \neq 2$), the cocycle $\alpha \defeq 1 \tensor \dual{y_2} \in S^2$ represents (a non-zero multiple of) the image of $\xi$ in $\Coho 2 (E; \QQ)$.
  By \cref{prop:transfer_circle_composite}, the Becker--Gottlieb transfer $(f \after g)^! \colon \Coho * (E_\xi; \QQ) \to \Coho * (B; \QQ)$ factors as
  \[ \Coho * (E_\xi; \QQ)  \longto  \Coho {* - 1}(E; \QQ)  \longto  \Coho {*} (B; \QQ) \]
  where the first map is the Gysin homomorphism and the second map is induced by (a non-zero multiple of) the map
  \[ \shift S  \longto  \k,  \qquad  \shift x  \longmapsto  - \tr^{\shift S}_\k \bigl( \mu^S_3(\shift x, \blank, \shift \alpha) \bigr) \]
  of cochain complexes.
  Note that the image of the Gysin homomorphism is the kernel of multiplication with $\xi$ on the cohomology of $X$, and that $1 \tensor \dual{y_n}$ represents a cohomology class in that kernel.
  By \eqref{eq:example_mu_3}, we have
  \[ \tr^{\shift S}_\k \bigl( \mu^S_3 \bigl( \shift (1 \tensor \dual{y_n}), \blank, \shift \alpha \bigr) \bigr)  =  \pm \dual z \]
  and hence the transfer $(f \after g)^!$ is non-trivial on rational cohomology.
\end{proof}

Altogether, we have in particular proven the following.

\begin{theorem} \label{thm:main}
  The Becker--Gottlieb transfer is not functorial.
  In fact, it is not even functorial on rational cohomology.
  That is, there exist two maps of animas with compact fibers
  \[ X  \xlongto{g}  Y  \xlongto{f}  Z \]
  such that the maps of spectra
  \[ g^! \after f^! \colon \SS[Z]  \longto  \SS[X]  \displayand  (f \after g)^! \colon \SS[Z]  \longto  \SS[X]  \]
  induce different maps on rational cohomology.
\end{theorem}

\begin{proof}
  By \cref{prop:example}, there exist maps of animas
  \[ E_\xi  \xlongto{g}  E  \xlongto{f}  \Sphere{n+1} \]
  such that the Becker--Gottlieb transfer $(f \after g)^!$ is non-trivial on rational cohomology.
  The map $E \to \Sphere{n+1}$ has fibers $\Sphere 2 \vee \Sphere n$, and $g \colon E_\xi \to E$ is the oriented $\Sphere 1$-bundle classified by some cohomology class $\xi \in \Coho 2 (E; \ZZ)$.
  But the Becker--Gottlieb transfer $g^!$ is trivial by \cref{prop:BG_circle}.
\end{proof}

\begin{appendices}

\crefalias{section}{appsec}

\include{TCinf}

\end{appendices}

\printbibliography

\end{document}

%% file: TCinf.tex
\section{Modeling odd spherical fibrations using \texorpdfstring{$\Cinf$}{C\_infty}-algebras} \label{app:Cinf_circle}

In this appendix, we construct explicit $\Cinf$-models for odd spherical fibrations (including circle bundles), and use this to give an alternative proof of \cref{prop:transfer_circle_composite}.

\subsection*{Recollections}

In this subsection, we recall the definition of $\Ainf$-algebras (originally due to Stasheff \cite{Sta}) and $\Cinf$-algebras (which first appeared in work of Smirnov \cite{Smi}; also see Kadeishvilli \cite{Kad}).
We will follow the conventions of \cite[§3]{II} (which are taken from \cite{Kel}).

\begin{notation}
  Throughout this subsection, we fix a cdga $\k$.
\end{notation}

\begin{definition}
  An $\Ainf$-algebra structure on a $\k$-module $R$ consists of a collection $(\mu^R_n)_{n \ge 1}$ of $\k$-linear maps of degree $1$
  \[ \mu^R_n \colon (\shift R)^{\otimes_\k n} \longto \shift R \]
  such that $\mu^R_1$ is the differential of $\shift R$ (i.e.\ $\mu_1^R(\shift x) = - \shift d_R(x)$) and the equation
  \begin{equation} \label{eq:Ainf}
    \sum_{r+s+t = n} \mu^R_{r+1+t} \circ ({\id^{\otimes r}} \otimes \mu^R_s \otimes \id^{\otimes t}) = 0
  \end{equation}
  holds for all $n \ge 1$.
  
  A morphism $f \colon R \to S$ of $\Ainf$-algebras over $\k$ is a collection $(f_n)_{n \ge 1}$ of $\k$-linear maps of degree $0$
  \[ f_n \colon (\shift R)^{\otimes_\k n} \longto \shift S\]
  such that
  \begin{equation} \label{eq:Ainf_map}
    \sum_{n_1 + \dots + n_i = n} \mu^{S}_i \circ (f_{n_1} \otimes \dots \otimes f_{n_i}) = \sum_{r+s+t = n} f_{r+1+t}\circ ({\id^{\otimes r}} \otimes \mu^R_s \otimes \id^{\otimes t})
  \end{equation}
  holds for all $n \ge 1$.
  The structure maps of the composite of two morphisms $f \colon R \to S$ and $g \colon S \to S'$ are defined to be
  \[ (g \circ f)_n  =  \sum_{n_1 + \dots + n_i = n} g_i \circ (f_{n_1} \otimes \dots \otimes f_{n_i}) \]
  and the identity $\id \colon R \to R$ has structure maps $\id_1 = \id$ and $\id_n = 0$ for $n \ge 2$.
\end{definition}

\begin{definition}
  We say that a morphism $f \colon R \to S$ of $\Ainf$-algebras over $\k$ is a \emph{quasi-isomorphism} when its linear part $f_1 \colon \shift R \to \shift S$ is a quasi-isomorphism.
  We say that $f$ is \emph{strict} when $f_n = 0$ for all $n \ge 2$.
\end{definition}

\begin{definition}
  An $\Ainf$-algebra $R$ over $\k$ is \emph{unital} if it is equipped with a distinguished cocycle $1 \in R^0$ such that
  \[ \mu_2(\shift 1, \shift a) = \shift a = (-1)^{\deg {a}} \mu_2(\shift a, \shift 1) \]
  for all $a \in R$ and such that $\mu_n(x_1, \dots, x_n) = 0$ whenever $n \ge 3$ and $x_i = \shift 1$ for some $i$.
  
  Let $R$ and $S$ be unital $\Ainf$-algebras over $\k$; a map $f \colon R \to S$ is \emph{unital} if
  \[ f_1(\shift 1) = \shift 1  \qquad \text{and} \qquad  f_n(x_1, \dots, x_n) = 0 \]
  whenever $n \ge 2$ and $x_i = \shift 1$ for some $i$.
\end{definition}

\begin{observation} \label{obs:dga_as_Ainf}
  Let $\k \to R$ be a morphism of dgas.
  Then $R$ becomes a unital $\Ainf$-algebra over $\k$ by defining
  \[ \mu_1(\shift a) \defeq - \shift d(a), \qquad \mu_2(\shift a, \shift b) \defeq (-1)^{\deg a} \shift (a \cdot b), \]
  and $\mu_n \defeq 0$ for $n \ge 3$.
  A map $f \colon R \to S$ of dgas over $\k$ yields a strict, unital map of unital $\Ainf$-algebras over $\k$ by defining $f_1(\shift a) \defeq \shift f(a)$ and $f_n \defeq 0$ for $n > 1$.
\end{observation}

We now recall the definition of $\Cinf$-algebras, which are a (strictly) commutative analog of $\Ainf$-algebras.

\begin{definition} \label{def:shuffle}
  Let $p$ and $q$ be non-negative integers.
  A \emph{$(p, q)$-shuffle} is a permutation $\sigma \in \Sigma_{p+q}$ such that the inequalities
  \[ \sigma(1) < \sigma(2) < \dots < \sigma(p)  \qquad \text{and} \qquad  \sigma(p + 1) < \sigma(p + 2) < \dots < \sigma(p + q) \]
  hold.
  We denote by $\Sh p q \subseteq \Sigma_{p + q}$ the set of $(p, q)$-shuffles.
\end{definition}

\begin{definition}
  A \emph{$\Cinf$-algebra over $\k$} is an $\Ainf$-algebra $R$ over $\k$ such that its structure maps vanish on sums of shuffles, i.e.\ for all $p,q \ge 1$ the \emph{shuffle relation}
  \begin{equation} \label{eq:Cinf_shuffle}
    \sum_{\sigma \in \Sh p q} \mu^R_{p+q} \after \sigma = 0
  \end{equation}
  holds, where $\sigma$ acts from the left, i.e.\ $\sigma \act (x_1 \tensor \dots \tensor x_{p+q}) \defeq \pm x_{\inv \sigma(1)} \tensor \dots \tensor x_{\inv \sigma(p+q)}$ with the sign determined by the Koszul rule.
  A map $R \to S$ of $\Cinf$-algebras over $\k$ is a map $f$ of the underlying $\Ainf$-algebras over $\k$ such that the \emph{shuffle relation}
  \begin{equation} \label{eq:Cinf_map_shuffle}
    \sum_{\sigma \in \Sh p q} f_{p+q} \after \sigma = 0
  \end{equation}
  holds for all $p, q \ge 1$.
\end{definition}

\begin{observation} \label{obs:cdga_as_Cinf}
  Let $\k \to R$ be a morphism of cdgas.
  Then $R$, considered as an $\Ainf$-algebra over $\k$ via \cref{obs:dga_as_Ainf}, is a unital $\Cinf$-algebra over $\k$.
  Similarly a map $f \colon R \to S$ of cdgas over $\k$ yields a strict, unital map of $\Cinf$-algebra over $\k$.
\end{observation}

\subsection*{The construction}

In this subsection, given a $\Cinf$-algebra $R$ and a cocycle $\alpha \in R^d$ for $d$ even, we construct a new $\Cinf$-algebra $R_\alpha$.
When $R$ models a space $X$, we prove in \cref{thm:Cinf_circle} below that $R_\alpha$ models the principal $\K(\QQ, d-1)$-fibration over $X$ classified by the cohomology class represented by $\alpha$.
This is based on the fact that, when $R$ is a Sullivan algebra, the construction $R_\alpha$ recovers the well-known Sullivan model for such fibrations (see e.g.\ \cite[§15(b), Example~4]{FHT}), combined with a certain functoriality of the construction $R_\alpha$, which we prove in \cref{con:Cinf_circle_functor}.
Also see \cref{rem:Cinf_circle} for a way of coming up with the formula for $R_\alpha$ without having to guess.
In the case that $\alpha = 0$, this essentially recovers a special case of a result of Lupton--Smith \cite[Theorem~2.1]{LS07}.

\begin{proposition} \label{con:Cinf_circle}
  Let $d$ be an even integer, $\k$ a cdga, $R$ a $\Cinf$-algebra over $\k$, and $\alpha \in R^d$ a cocycle.
  Then the following construction equips $R_\alpha \defeq R \tensor \QQ \langle 1, a \rangle$, where $\deg a = d-1$, with the structure of a $\Cinf$-algebra over $\k$.
  
  For all $1 \le i \le n$, let $x_i \in R$ and $b_i \in \{1, a\}$; furthermore let $j_1 < \dots < j_m$ be those indices $i$ such that $b_i = a$.
  Then we define
  \begin{multline*}
      \mu_n \bigl( \shift (x_1 \tensor b_1), \dots, \shift (x_n \tensor b_n) \bigr) \\
      \defeq  \sum_{k_\bullet} (-1)^{\epsilon_k} \mu^R_{n + m}(\shift x_1, \dots, \shift x_{k_1}, \shift \alpha, \shift x_{k_1 + 1}, \dots, \shift x_{k_m}, \shift \alpha, \shift x_{k_m + 1}, \dots, \shift x_n) \tensor 1 \\
      + \sum_{l_\bullet} (-1)^{\epsilon_l} \mu^R_{n + m - 1}(\shift x_1, \dots, \shift x_{l_1}, \shift \alpha, \shift x_{l_1 + 1}, \dots, \shift x_{l_{m-1}}, \shift \alpha, \shift x_{l_{m-1} + 1}, \dots, \shift x_n) \tensor a
  \end{multline*}
  where $k_\bullet$ runs over those sequences $k_1 < \dots < k_m$ such that $j_i \le k_i < j_{i+1}$, and $l_\bullet$ runs over those sequences $l_1 < \dots < l_{m-1}$ such that $j_i \le l_i < j_{i+1}$; the signs are given by
  \[ \epsilon_k  \defeq  \sum_{1 \le i \le m} \sum_{j_i < p \le k_i} \deg{\shift x_p}  \qquad \text{and} \qquad  \epsilon_l  \defeq  \sum_{1 \le i \le m - 1} \sum_{j_i < p \le l_i} \deg{\shift x_p} + \sum_{j_m < p \le n} \deg{\shift x_p} \]
  (we set $j_{m+1} \defeq n + 1$ throughout).
  When $m = 0$, by convention there is no sequence $l_\bullet$ and only the empty sequence $k_\bullet$; hence we have
  \[ \mu_n \bigl( \shift (x_1 \tensor 1), \dots, \shift (x_n \tensor 1) \bigr)  =  \mu^R_n(\shift x_1, \dots, \shift x_n) \tensor 1 \]
  i.e.\ the map $\shift x \mapsto \shift (x \tensor 1)$ defines a strict map $\iota_R \colon R \to R_\alpha$ of $\Cinf$-algebras.
\end{proposition}

\begin{proof}
  Observe that $R_\alpha$ is indeed a $\k$-module, and that the maps $\mu_n$ are $\k$-multilinear.
  Then the claim follows from \cref{lemma:Cinf_circle_shuffle,lemma:Cinf_circle} below.
\end{proof}

\begin{observation}
  In the situation of \cref{con:Cinf_circle}, we in particular have
  \[ \mu_1 \bigl( \shift (y \tensor a) \bigr)  =  \mu^R_2(\shift y, \shift \alpha) \tensor 1 + \mu^R_1(\shift y) \tensor a \]
  in arity $1$, and
  \begin{align*}
    \mu_2 \bigl( \shift (x \tensor 1), \shift(y \tensor a) \bigr) &= \mu^R_3(\shift x, \shift y, \shift \alpha) \tensor 1 + \mu^R_2(\shift x, \shift y) \tensor a \\
    \mu_2 \bigl( \shift (x \tensor a), \shift(y \tensor 1) \bigr) &= \mu^R_3(\shift x, \shift \alpha, \shift y) \tensor 1 + (-1)^{\deg{\shift y}} \mu^R_3(\shift x, \shift y, \shift \alpha) \tensor 1 + (-1)^{\deg{\shift y}} \mu^R_2(\shift x, \shift y) \tensor a \\
    \mu_2 \bigl( \shift (y_1 \tensor a), \shift(y_2 \tensor a) \bigr) &= \mu^R_4(\shift y_1, \shift \alpha, \shift y_2, \shift \alpha) \tensor 1 + \mu^R_3(\shift y_1, \shift \alpha, \shift y_2) \tensor a
  \end{align*}
  in arity $2$, as well as
  \begin{multline*}
    \mu_n \bigl( \shift (y_1 \tensor a), \dots, \shift(y_n \tensor a) \bigr) \\
    = \mu_{2n}(\shift y_1, \shift \alpha, \shift y_2, \shift \alpha, \dots, \shift y_n, \shift \alpha) \tensor 1 + \mu_{2n-1}(\shift y_1, \shift \alpha, \shift y_2, \shift \alpha, \dots, \shift y_n) \tensor a
  \end{multline*}
  in any arity $n$.
\end{observation}

\begin{lemma} \label{lemma:Cinf_circle_shuffle}
  Let $d$ be an even integer, $R$ a $\Cinf$-algebra, and $\alpha \in R^d$ a cocycle.
  Then the maps $\mu_n$ of \cref{con:Cinf_circle} fulfill the shuffle relations.
\end{lemma}

\begin{proof}
  Let $p$ and $q$ be positive integers, and $n \defeq p + q$.
  Then the $(p,q)$-shuffle relation \eqref{eq:Cinf_shuffle} for $\mu_n$ reads
  \begin{equation} \label{eq:Cinf_circle_shuffle}
    \sum_{\sigma \in \Sh{p}{q}} \pm \mu_n \bigl( \shift (x_{\inv{\sigma}(1)} \tensor b_{\inv{\sigma}(1)}), \dots, \shift (x_{\inv{\sigma}(n)} \tensor b_{\inv{\sigma}(n)}) \bigr) = 0
  \end{equation}
  for all $x_i \in R$ and $b_i \in \{1, a\}$.
  Let $j_1 < \dots < j_m$ be those indices $1 \le i \le p$ such that $b_i = a$, and let $j_{m+1} < \dots < j_{m+m'}$ be those indices $p + 1 \le i \le n$ such that $b_i = a$.

  We observe that the part of the left-hand side of \eqref{eq:Cinf_circle_shuffle} that lands in $R \tensor 1$ is equal to
  \[ \sum_{k_\bullet} \pm \sum_{\sigma \in S} \mu^R_{n+m+m'} ( \sigma \act y_\bullet ) \tensor 1 \]
  where
  \begin{equation*}
    y_\bullet  \defeq  (\shift x_1, \dots, \shift x_{k_1}, \shift \alpha, \shift x_{k_1 + 1}, \dots, \shift x_{k_{m+m'}}, \shift \alpha, \shift x_{k_{m+m'} + 1}, \dots, \shift x_n)
  \end{equation*}
  and $k_\bullet$ runs over all sequences $k_1 < \dots < k_{m+m'}$ such that $j_i \le k_i < j_{i+1}$ (where we set $j_{m + m' + 1} \defeq n + 1$) and such that $k_i \le p$ for $i \le m$.
  The set $S \subseteq \Sh {p+m} {q+m'}$ consists of those shuffles $\sigma$ such that
  \begin{equation} \label{eq:sigma}
    \sigma(k_s + s) < \sigma(j_t + t - 1)  \qquad \text{or} \qquad  \sigma(k_t + t) < \sigma(j_s + s - 1) 
  \end{equation}
  for all $1 \le s \le m$ and $m + 1 \le t \le m + m'$.
  
  We will now prove that $\sum_{\sigma \in S} \mu^R_{n+m+m'} ( \sigma \act y_\bullet ) = 0$ for any $k_\bullet$.
  To this end, we define a free involution $\tau$ on $\Sh {p+m} {q+m'} \setminus S$.
  Let $\sigma \in \Sh {p+m} {q+m'} \setminus S$, and let $(s, t)$ be the lexicographically largest tuple for which \eqref{eq:sigma} does not hold.
  Then we define $\tau(\sigma)(k_s + s) \defeq \sigma(k_t + t)$, $\tau(\sigma)(k_t + t) \defeq \sigma(k_s + s)$ and $\tau(\sigma)(i) \defeq \sigma(i)$ otherwise.
  One checks that $\tau(\sigma)$ is again a $(p+m, q+m')$-shuffle, and that $(s, t)$ is still the lexicographically largest tuple for which \eqref{eq:sigma} does not hold for $\tau(\sigma)$.
  Hence $\tau$ is in fact an involution.
  Since $\mu^R_{n+m+m'} ( \tau(\sigma) \act y_\bullet ) = - \mu^R_{n+m+m'} ( \sigma \act y_\bullet )$, it follows that
  \[ \sum_{\sigma \in S} \mu^R_{n+m+m'} ( \sigma \act y_\bullet )  =  \sum_{\sigma \in \Sh {p+m} {q+m'}} \mu^R_{n+m+m'} ( \sigma \act y_\bullet ) \]
  which vanishes by the corresponding shuffle relation of $R$.

  The part of the left-hand side of \eqref{eq:Cinf_circle_shuffle} that lands in $R \tensor a$ is equal to
  \begin{equation} \label{eq:shuffle_a}
    \sum_{k_\bullet, l_\bullet} \pm \sum_{\sigma \in S} \mu^R_{n+m+m'-1} \bigl( \sigma \act (y_\bullet, z_\bullet) \bigr) \tensor a + \sum_{k'_\bullet, l'_\bullet} \pm \sum_{\sigma' \in S'} \mu^R_{n+m+m'-1} \bigl( \sigma' \act (y'_\bullet, z'_\bullet) \bigr) \tensor a
  \end{equation}
  where
  \begin{align*}
    y_\bullet  &\defeq  (\shift x_1, \dots, \shift x_{k_1}, \shift \alpha, \shift x_{k_1 + 1}, \dots, \shift x_{k_m}, \shift \alpha, \shift x_{k_m + 1}, \dots, \shift x_p) \\
    z_\bullet  &\defeq  (\shift x_{p+1}, \dots, \shift x_{l_{m+1}}, \shift \alpha, \shift x_{l_{m+1} + 1}, \dots, \shift x_{l_{m+m'-1}}, \shift \alpha, \shift x_{l_{m+m'-1} + 1}, \dots, \shift x_n) \\
    y'_\bullet  &\defeq  (\shift x_1, \dots, \shift x_{k'_1}, \shift \alpha, \shift x_{k'_1 + 1}, \dots, \shift x_{k'_{m-1}}, \shift \alpha, \shift x_{k'_{m-1} + 1}, \dots, \shift x_p) \\
    z'_\bullet  &\defeq  (\shift x_{p+1}, \dots, \shift x_{l'_{m+1}}, \shift \alpha, \shift x_{l'_{m+1} + 1}, \dots, \shift x_{l'_{m+m'}}, \shift \alpha, \shift x_{l'_{m+m'} + 1}, \dots, \shift x_n)
  \end{align*}
  and $k_\bullet$, $l_\bullet$, $k'_\bullet$, and $l'_\bullet$ run over all strictly increasing sequences of the appropriate length such that $j_i \le k_i < \min(j_{i+1}, p + 1)$ and the same for $k'_\bullet$, and $j_i \le l_i < j_{i+1}$ and the same for $l'_\bullet$ (here we set $j_{m+m'+1} \defeq n+1$).
  The set $S \subseteq \Sh {p+m} {q+m'-1}$ consists of those shuffles $\sigma$ such that
  \[ \sigma(k_s + s) < \sigma(j_t + t - 1)  \qquad \text{or} \qquad  \sigma(l_t + t) < \sigma(j_s + s - 1) \]
  for all $1 \le s \le m$ and $m + 1 \le t \le m + m'$, ignoring any inequality involving $l_{m + m'}$.
  Similarly, the set $S' \subseteq \Sh {p+m-1} {q+m'}$ consists of those shuffles $\sigma'$ such that
  \[ \sigma'(k_s + s) < \sigma'(j_t + t - 2)  \qquad \text{or} \qquad  \sigma'(l_t + t - 1) < \sigma'(j_s + s - 1) \]
  for all $1 \le s \le m$ and $m + 1 \le t \le m + m'$, ignoring any inequality involving $k'_m$.
  
  % Consider some $k_\bullet$, $l_\bullet$, and $\sigma \in S$ such that $\sigma(j_{m+m'} + m + m' - 1) < \sigma(k_m + m)$.
  % Then we can construct a $k'_\bullet$, $l'_\bullet$, and $\sigma' \in S'$ such that $\sigma(j_m + m - 1) < \sigma(l'_{m+m'} + m + m' - 1)$: define $k'_i \defeq k_i$, $l'_{m+m'} \defeq \sigma(k_m + m) - m - m' + 1$
  We claim that \eqref{eq:shuffle_a} is equal to
  \[ \sum_{k_\bullet, l_\bullet} \pm \sum_{\sigma \in \Sh {p+m} {q+m'-1}} \mu^R_{n+m+m'-1} \bigl( \sigma \act (y_\bullet, z_\bullet) \bigr) \tensor a + \sum_{k'_\bullet, l'_\bullet} \pm \sum_{\sigma' \in \Sh {p+m-1} {q+m'}} \mu^R_{n+m-1+m'} \bigl( \sigma' \act (y'_\bullet, z'_\bullet) \bigr) \tensor a \]
  which vanishes by the shuffle relations of $R$.
  To see this, we observe that, similarly to above, the set
  \[ \coprod_{k_\bullet,l_\bullet} (\Sh {p+m} {q+m'-1} \setminus S)  \cop  \coprod_{k'_\bullet,l'_\bullet} (\Sh {p+m-1} {q+m'} \setminus S') \]
  admits a free involution such that the two terms of the sum above corresponding to any one orbit exactly cancel each other.
\end{proof}

\begin{lemma} \label{lemma:Cinf_circle}
    Let $d$ be an even integer, $R$ a $\Cinf$-algebra, and $\alpha \in R^d$ a cocycle.
    Then the maps $\mu_n$ of \cref{con:Cinf_circle} fulfill the $\Ainf$-algebra relations.
\end{lemma}

\begin{proof}
  The $n$-th $\Ainf$-algebra relation \eqref{eq:Ainf} for $R_\alpha$ reads
  \begin{equation} \label{eq:Cinf_circle_Ainf}
    \sum_{1 \le l \le r \le n} (-1)^{\epsilon_l} \mu_{n+l-r} \bigl( \shift (x_1 \tensor b_1), \dots, \mu_{r-l+1} \bigl( \shift (x_l \tensor b_l), \dots, \shift (x_r \tensor b_r) \bigr), \dots, \shift (x_n \tensor b_n) \bigr) = 0
  \end{equation}
  where $x_i \in R$, $b_i \in \{1, a\}$, and $\epsilon_l \defeq \sum_{i = 1}^{l-1} \deg{\shift (x_1 \tensor b_1)}$.
  When all $b_i$ are equal to $1$, this expands to the $n$-th $\Ainf$-algebra relation for $R$.
  Hence we can assume that at least one $b_i$ is equal to $a$; let $j_1 < \dots < j_m$ be those indices $i$ such that $b_i = a$.
  We will consider the term of the sum above corresponding to some $1 \le l \le r \le n$.
  
  Let $j'_1 < \dots < j'_{m'}$ be those $j_i$ such that $l \le j_i \le r$.
  Then
  \[ \mu_{r-l+1} \bigl( \shift (x_l \tensor b_l), \dots, \shift (x_r \tensor b_r) \bigr) \]
  is a sum of all terms of the form
  \begin{equation*}
    \pm \mu^R_{r-l+1 + m'}(v^1_\bullet) \tensor 1  \displayand  \pm \mu^R_{r - l + m'}(v^a_\bullet) \tensor a
  \end{equation*}
  where
  \begin{align*}
    v^1_\bullet &\defeq (\shift x_l, \dots, \shift x_{k_1}, \shift \alpha, \shift x_{k_1 + 1}, \dots, \shift x_{k_{m'}}, \shift \alpha, \shift x_{k_{m'} + 1}, \dots, \shift x_r) \\
    v^a_\bullet &\defeq (\shift x_l, \dots, \shift x_{k_1}, \shift \alpha, \shift x_{k_1 + 1}, \dots, \shift x_{k_{m'-1}}, \shift \alpha, \shift x_{k_{m'-1} + 1}, \dots, \shift x_r)
  \end{align*}
  where in either case $k_\bullet$ runs over all strictly increasing sequences of the appropriate length such that $j'_i \le k_i < j'_{i+1}$ (where we set $j'_{m' + 1} \defeq r + 1$); when $m' = 0$, we omit the terms involving $v^a_\bullet$.

  Hence the left-hand side of \eqref{eq:Cinf_circle_Ainf} is the sum of all terms of the form
  \begin{equation} \label{eq:Cinf_circle_Ainf_terms}
    \begin{aligned}
      \pm \mu^R_{n+l-r+m-m'}(u_\bullet, \mu^R_{r-l+1 + m'}(v^1_\bullet), w^{11}_\bullet) &\tensor 1 \\
      \pm \mu^R_{n+l-r+m-m'}(\hat u_\bullet, \mu^R_{r-l+1 + m'}(v^1_\bullet), w^{a1}_\bullet) &\tensor 1 \\
      \pm \mu^R_{n+l-r+m-m'+1}(u_\bullet, \mu^R_{r-l + m'}(v^a_\bullet), w^{a1}_\bullet) &\tensor 1 \\
      \pm \mu^R_{n+l-r+m-m'-1}(u_\bullet, \mu^R_{r-l+1 + m'}(v^1_\bullet), w^{1a}_\bullet) &\tensor a \\
      \pm \mu^R_{n+l-r+m-m'-1}(\hat u_\bullet, \mu^R_{r-l+1 + m'}(v^1_\bullet), w^{aa}_\bullet) &\tensor a \\
      \pm \mu^R_{n+l-r+m-m'}(u_\bullet, \mu^R_{r-l + m'}(v^a_\bullet), w^{aa}_\bullet) &\tensor a
    \end{aligned}
  \end{equation}
  whose ingredients we now define.
  The terms involving $\hat u_\bullet$ correspond to one $\shift \alpha$ ``jumping over'' the inner $\mu^R$, and hence only exist when $j_1 < l$.
  We let $s$ denote the largest $i$ such that $j_i < l$ (or $0$ if none such exists), and define
  \begin{align*}
    u_\bullet &\defeq (\shift x_1, \dots, \shift x_{k_1}, \shift \alpha, \shift x_{k_1 + 1}, \dots, \shift x_{k_s}, \shift \alpha, \shift x_{k_s + 1}, \dots, \shift x_{l-1}) \\
    \hat u_\bullet &\defeq (\shift x_1, \dots, \shift x_{k_1}, \shift \alpha, \shift x_{k_1 + 1}, \dots, \shift x_{k_{s-1}}, \shift \alpha, \shift x_{k_{s-1} + 1}, \dots, \shift x_{l-1})
  \end{align*}
  where in both cases $k_\bullet$ runs over over all strictly increasing sequences of the appropriate length such that $j_i \le k_i < j_{i+1}$.
  We furthermore let $t$ denote the smallest $i$ such that $j_i > r$ (or $m + 1$ if none such exists), and define
  \begin{align*}
    w^{11}_\bullet &\defeq (\shift x_{r+1}, \dots, \shift x_{k_t}, \shift \alpha, \shift x_{k_t + 1}, \dots, \shift x_{k_m}, \shift \alpha, \shift x_{k_m + 1}, \dots, \shift x_n) \\
    w^{a1}_\bullet &\defeq (\shift x_{r+1}, \dots, \shift x_{k_{t-1}}, \shift \alpha, \shift x_{k_{t-1} + 1}, \dots, \shift x_{k_m}, \shift \alpha, \shift x_{k_m + 1}, \dots, \shift x_n) \\
    w^{1a}_\bullet &\defeq (\shift x_{r+1}, \dots, \shift x_{k_t}, \shift \alpha, \shift x_{k_t + 1}, \dots, \shift x_{k_{m-1}}, \shift \alpha, \shift x_{k_{m-1} + 1}, \dots, \shift x_n) \\
    w^{aa}_\bullet &\defeq (\shift x_{r+1}, \dots, \shift x_{k_{t-1}}, \shift \alpha, \shift x_{k_{t-1} + 1}, \dots, \shift x_{k_{m-1}}, \shift \alpha, \shift x_{k_{m-1} + 1}, \dots, \shift x_n)
  \end{align*}
  where in each case $k_\bullet$ runs over all strictly increasing sequences of the appropriate length such that $\max(j_i, r) \le k_i < j_{i+1}$ (where $j_{m+1} \defeq n + 1$).

  All terms as in \eqref{eq:Cinf_circle_Ainf_terms} that do not involve $\hat u_\bullet$ are of the forms
  \begin{align}
    (-1)^{\epsilon_p} \mu^R_{p + n + m - q} ( y^1_1, \dots, \mu^R_{q - p + 1} (y^1_p, \dots, y^1_q), \dots, y^1_{n + m}) &\tensor 1 \label{eq:y_1} \\
    (-1)^{\epsilon_{p'}} \mu^R_{p' + n + m - 1 - q'} ( y^a_1, \dots, \mu^R_{q' - p' + 1} (y^a_{p'}, \dots, y^a_{q'}), \dots, y^a_{n + m - 1}) &\tensor a \label{eq:y_a}
  \end{align}
  for some $1 \le p \le q \le n + m$ and $1 \le p' \le q' \le n + m - 1$, where
  \begin{align*}
    y^1_\bullet &\defeq (\shift x_1, \dots, \shift x_{k_1}, \shift \alpha, \shift x_{k_1 + 1}, \dots, \shift x_{k_m}, \shift \alpha, \shift x_{k_m + 1}, \dots, \shift x_n) \\
    y^a_\bullet &\defeq (\shift x_1, \dots, \shift x_{k_1}, \shift \alpha, \shift x_{k_1 + 1}, \dots, \shift x_{k_{m-1}}, \shift \alpha, \shift x_{k_{m-1} + 1}, \dots, \shift x_n)
  \end{align*}
  where in either case $k_\bullet$ runs over all strictly increasing sequences of the appropriate length such that $j_i \le k_i < j_{i+1}$ (where $j_{m+1} \defeq n + 1$); the signs are given by $\epsilon_p \defeq \sum_{i = 1}^{p-1} \deg {y_i}$ and analogously for $\epsilon_{p'}$.
  We note that the sum over \emph{all} terms of the form \eqref{eq:y_1} (not just those arising from \eqref{eq:Cinf_circle_Ainf_terms}) is zero by the $\Ainf$-algebra relation for $R$, and similarly for \eqref{eq:y_a}.

  We will now rewrite the sum of the terms of \eqref{eq:Cinf_circle_Ainf_terms} that do involve $\hat u_\bullet$ as a sum of terms of the forms \eqref{eq:y_1} and \eqref{eq:y_a}.
  First note that, when $m' = 0$, the $\hat u_\bullet$-terms of \eqref{eq:Cinf_circle_Ainf_terms} are already of the required form; hence we will now assume that $m' > 0$.
  In that case, \cref{lemma:Cinf_even_eq_2} below implies that
  \[ \sum_{k_\bullet} \pm \mu^R_{r-l+1 + m'}(v^1_\bullet)  =  \sum_{k'_\bullet} \pm \mu^R_{r-l+1 + m'}(\hat v^1_\bullet) \]
  where
  \[ \hat v^1_\bullet  \defeq  (\shift x_l, \dots, \shift x_{k'_0}, \shift \alpha, \shift x_{k'_0 + 1}, \dots, \shift x_{k'_{m'-1}}, \shift \alpha, \shift x_{k_{m'-1} + 1}, \dots, \shift x_r) \]
  where $k'_\bullet$ runs over all sequences $k'_0 < \dots < k'_{m'-1}$ such that $j'_i \le k'_i < j'_{i+1}$ (where $j'_0 \defeq l - 1$).
  We observe that
  \begin{align*}
    \pm \mu^R_{n+l-r+m-m'}(\hat u_\bullet, \mu^R_{r-l+1 + m'}(\hat v^1_\bullet), w^{a1}_\bullet) &\tensor 1 \\
    \pm \mu^R_{n+l-r+m-m'-1}(\hat u_\bullet, \mu^R_{r-l+1 + m'}(\hat v^1_\bullet), w^{aa}_\bullet) &\tensor a
  \end{align*}
  are of the forms \eqref{eq:y_1} and \eqref{eq:y_a}, respectively (in particular the signs match, as a diligent check shows).

  We now note that, after replacing the $\hat u_\bullet$-terms of \eqref{eq:Cinf_circle_Ainf_terms} in this way, each term of the forms \eqref{eq:y_1} and \eqref{eq:y_a} appears at most once.
  For a fixed $y^1_\bullet$ or $y^a_\bullet$, corresponding to some sequence $k_\bullet$, the terms of the form \eqref{eq:y_1} or \eqref{eq:y_a} that do not appear in \eqref{eq:Cinf_circle_Ainf_terms} are precisely those such that there both exists an $i$ such that $j_i + i \le p \le k_i + i \le q$ and there does not exist an $i'$ such that $p \le j_{i'} + i' - 1 \le q < k_{i'} + i'$.
  The sum of all of these terms is zero by \cref{lemma:Cinf_even_eq} below.
  Hence the sum of all terms of the forms \eqref{eq:Cinf_circle_Ainf_terms} is zero, and thus \eqref{eq:Cinf_circle_Ainf} holds.
\end{proof}

\begin{lemma} \label{lemma:Cinf_even_eq_2}
  Let $R$ be a $\Cinf$-algebra and $\alpha \in R$ an element of even degree.
  Furthermore let $1 \le j_1 < \dots < j_m \le n$ be integers (with $m$ positive), and $y_i \in \shift R$ for $1 \le i \le n$.
  Then
  \begin{equation} \label{eq:Cinf_even_eq_2}
  \begin{multlined}
    \sum_{k'_\bullet} (-1)^{\epsilon_{k'}} \mu^R_{n + m + 1}(y_1, \dots, y_{k'_0}, \shift \alpha, y_{k'_0 + 1}, \dots, y_{k'_{m-1}}, \shift \alpha, y_{k'_{m-1} + 1}, \dots, y_n) \\
     = - \sum_{k_\bullet} (-1)^{\epsilon_{k}} \mu^R_{n + m + 1}(y_1, \dots, y_{k_1}, \shift \alpha, y_{k_1 + 1}, \dots, y_{k_m}, \shift \alpha, y_{k_m + 1}, \dots, y_n)
  \end{multlined}
  \end{equation}
  where $k'_\bullet$ runs over those sequences $k'_0 < \dots < k'_{m-1}$ such that $j_i \le k'_i < j_{i+1}$, similarly $k_\bullet$ runs over those sequences $k_1 < \dots < k_m$ such that $j_i \le k_i < j_{i+1}$, and the signs are given by
  \[ \epsilon_{k'}  \defeq  \sum_{0 \le i \le m-1} \sum_{j_i < p \le k'_i} \deg{y_p}  \displayand  \epsilon_{k}  \defeq  \sum_{1 \le i \le m-1} \sum_{j_i < p \le k_i} \deg{y_p} + \sum_{1 \le p \le k_m} \deg{y_p} + m - 1 \]
  (where $j_0 \defeq 0$ and $j_{m+1} \defeq n + 1$ throughout).
\end{lemma}

\begin{proof}
  Consider the $(1,n+m)$-shuffle relation \eqref{eq:Cinf_shuffle} applied to
  \[ \shift \alpha \tensor y_1 \tensor \dots \tensor y_{l_1} \tensor \shift \alpha \tensor y_{l_1 + 1} \tensor \dots \tensor y_{l_m} \tensor \shift \alpha \tensor y_{l_m + 1} \tensor \dots \tensor y_n \]
  where $l_\bullet$ is a sequence $l_1 < \dots < l_{m-1}$ such that $j_i \le l_i < j_{i+1}$.
  Taking the sum of these relations for all $l_\bullet$, each multiplied by $-1$ to the power of
  \[ \epsilon_l  \defeq  \sum_{1 \le i \le m-1} \sum_{j_i < p \le l_i} \deg{y_p} \]
  we note that each term of the two sums of \eqref{eq:Cinf_even_eq_2} appears exactly once, namely for the sequence $l_\bullet$ with $l_i = k'_i$ or $l_i = k_i$, respectively.
  All other terms in the sum appear exactly twice, with opposite sign (using that $\shift \alpha$ has odd degree).
\end{proof}

\begin{lemma} \label{lemma:Cinf_even_eq}
  Let $R$ be a $\Cinf$-algebra and $\alpha \in R$ an element of even degree.
  Furthermore let $1 \le j_1 < \dots < j_m \le n$ be integers ($m$ can be $0$), and $y_i \in \shift R$ for $1 \le i \le n$.
  Then
  \begin{equation} \label{eq:Cinf_even_eq}
    \sum_{k_\bullet} (-1)^{\epsilon_k} \mu^R_{n + m + 1}(y_1, \dots, y_{k_0}, \shift \alpha, y_{k_0 + 1}, \dots, y_{k_m}, \shift \alpha, y_{k_m + 1}, \dots, y_n) = 0
  \end{equation}
  where $k_\bullet$ runs over those sequences $k_0 < \dots < k_m$ such that $j_i \le k_i < j_{i+1}$, and the sign is given by
  \[ \epsilon_k  \defeq  \sum_{0 \le i \le m} \sum_{j_i < p \le k_i} \deg{y_p} \]
  (where $j_0 \defeq 0$ and $j_{m+1} \defeq n + 1$ throughout).
\end{lemma}

\begin{proof}
  Consider the $(1,n+m)$-shuffle relation \eqref{eq:Cinf_shuffle} applied to
  \[ \shift \alpha \tensor y_1 \tensor \dots \tensor y_{l_1} \tensor \shift \alpha \tensor y_{l_1 + 1} \tensor \dots \tensor y_{l_m} \tensor \shift \alpha \tensor y_{l_m + 1} \tensor \dots \tensor y_n \]
  where $l_\bullet$ is a sequence $l_1 < \dots < l_m$ such that $j_i \le l_i < j_{i+1}$.
  Taking the sum of these relations for all $l_\bullet$, each multiplied by $-1$ to the power of
  \[ \epsilon_l  \defeq  \sum_{1 \le i \le m} \sum_{j_i < p \le l_i} \deg{y_p} \]
  we note that each term of the left-hand side of \eqref{eq:Cinf_even_eq} appears exactly once, namely for the sequence $l_\bullet$ with $l_i = k_i$.
  All other terms in the sum appear exactly twice, with opposite sign (using that $\shift \alpha$ has odd degree).
\end{proof}

\begin{proposition} \label{con:Cinf_circle_functor}
  Let $d$ be an even integer, $\k$ a cdga, $f \colon R \to S$ a map of $\Cinf$-algebras over $\k$, and $\alpha \in R^d$ a cocycle.
  Then the following construction defines a map $f_\alpha \colon R_\alpha \to S_{f(\alpha)}$ of the $\Cinf$-algebras over $\k$ of \cref{con:Cinf_circle}.
  
  For all $1 \le i \le n$, let $x_i \in R$ and $b_i \in \{1, a\}$; furthermore let $j_1 < \dots < j_m$ be those indices $i$ such that $b_i = a$.
  Then we define
  \begin{multline*}
      (f_\alpha)_n \bigl( \shift (x_1 \tensor b_1), \dots, \shift (x_n \tensor b_n) \bigr) \\
      \defeq  \sum_{k_\bullet} (-1)^{\epsilon_k} f_{n + m}(\shift x_1, \dots, \shift x_{k_1}, \shift \alpha, \shift x_{k_1 + 1}, \dots, \shift x_{k_m}, \shift \alpha, \shift x_{k_m + 1}, \dots, \shift x_n) \tensor 1 \\
      + \sum_{l_\bullet} (-1)^{\epsilon_l} f_{n + m - 1}(\shift x_1, \dots, \shift x_{l_1}, \shift \alpha, \shift x_{l_1 + 1}, \dots, \shift x_{l_{m-1}}, \shift \alpha, \shift x_{l_{m-1} + 1}, \dots, \shift x_n) \tensor a
  \end{multline*}
  where $k_\bullet$ runs over those sequences $k_1 < \dots < k_m$ such that $j_i \le k_i < j_{i+1}$, and $l_\bullet$ runs over those sequences $l_1 < \dots < l_{m-1}$ such that $j_i \le l_i < j_{i+1}$; the signs are given by
  \[ \epsilon_k  \defeq  \sum_{1 \le i \le m} \sum_{j_i < p \le k_i} \deg{\shift x_p}  \qquad \text{and} \qquad  \epsilon_l  \defeq  \sum_{1 \le i \le m - 1} \sum_{j_i < p \le l_i} \deg{\shift x_p} + \sum_{j_m < p \le n} \deg{\shift x_p} \]
  (we set $j_{m+1} \defeq n + 1$ throughout).
  When $m = 0$, by convention there is no sequence $l_\bullet$ and only the empty sequence $k_\bullet$; hence we have
  \[ (f_\alpha)_n \bigl( \shift (x_1 \tensor 1), \dots, \shift (x_n \tensor 1) \bigr)  =  f_n(\shift x_1, \dots, \shift x_n) \tensor 1 \]
  and in particular $f_\alpha \after \iota_R = \iota_S \after f$, where $\iota_R \colon R \to R_\alpha$ and $\iota_S \colon S \to S_{f(\alpha)}$ are the strict maps of $\Cinf$-algebras given by $\shift x \mapsto \shift (x \tensor 1)$.
\end{proposition}

\begin{proof}
  First observe that the maps $(f_\alpha)_n$ are indeed $\k$-multilinear.
  That each map $(f_\alpha)_n$ fulfills the shuffle relations \eqref{eq:Cinf_map_shuffle} follows by verbatim the same proof as for \cref{lemma:Cinf_circle_shuffle}.
  The proof that the maps $(f_\alpha)_n$ assemble into a map of $\Ainf$-algebras is similar to the proof of \cref{lemma:Cinf_circle}, and we leave it to the reader.
\end{proof}

\begin{lemma} \label{lemma:Cinf_circle_functor_qiso}
  Let $d$ be an even integer, $f \colon R \to S$ a quasi-isomorphism of $\Cinf$-algebras, and $\alpha \in R^d$ a cocycle.
  Then the map $f_\alpha \colon R_\alpha \to S_{f(\alpha)}$ of \cref{con:Cinf_circle_functor} is again a quasi-isomorphism.
\end{lemma}

\begin{proof}
  Consider the commutative diagram of cochain complexes
  \[
  \begin{tikzcd}
    0 \rar & R \rar{\iota_R} \dar[swap]{f_1} & R_\alpha \rar \dar{(f_\alpha)_1} & R_\alpha / (R \tensor 1) \rar \dar & 0 \\
    0 \rar & S \rar{\iota_S} & S_{f(\alpha)} \rar & S_\alpha / (S \tensor 1) \rar & 0
  \end{tikzcd}
  \]
  with short exact rows.
  The map $f_1$ is a quasi-isomorphism by assumption.
  The induced maps on quotients is isomorphic to the map $f_1 \tensor a \colon R \tensor a \to S \tensor a$, which is also a quasi-isomorphism.
  Hence $(f_\alpha)_1$ is a quasi-isomorphism.
\end{proof}

\subsection*{Odd spherical fibrations}

We will now prove that principal $\K(\QQ, n)$-fibrations can, when $n$ is odd, be modeled by the construction of \cref{con:Cinf_circle}.
Note that an odd-dimensional sphere $\Sphere n$ is rationally equivalent to $\K(\QQ, n)$, and that in this case the classifying space $\B\aut^+(\Sphere n)$ of the topological monoid of orientation-preserving self--homotopy equivalences of $\Sphere n$ is rationally equivalent to $\K(\QQ, n + 1)$.
Hence any oriented $\Sphere n$-fibration is rationally a principal $\K(\QQ, n)$-fibration.

\begin{theorem} \label{thm:Cinf_circle}
  Let $d \ge 2$ be an even integer, $f \colon X \to B$ a map of nilpotent animas of finite rational type, $\xi$ an element of $\Coho d (X; \QQ)$, and $g \colon X_\xi \to X$ the principal $\K(\QQ, d-1)$-fibration classified by $\xi$.
  Furthermore let $\k$ be a cofibrant cdga that models $B$, let $S$ be a unital $\Cinf$-algebra over $\k$ that models $X$ over $B$, and let $\alpha \in S^d$ be a cocycle that represents $\xi$.
  Then the map $g \colon X_\xi \to X$ of animas over $B$ is modeled by the strict unital map $\iota_S \colon S \to S_\alpha$ of $\Cinf$-algebras over $\k$ constructed in \cref{con:Cinf_circle}.
\end{theorem}

\begin{proof}
  By assumption there exists a cofibration of cdgas $\k \to R$ that models $f$, and a unital quasi-isomorphism $\phi \colon S \to R$ of $\Cinf$-algebras.
  It is well known that the map $\iota_R \colon R \to R_{\phi(\alpha)}$ is a cofibration of cdgas that models $g$: when $X$ is simply connected (e.g.\ in the universal case $X = \K(\QQ, d)$), this is \cite[§15(b), Example~4]{FHT}; the more general case of $X$ being merely nilpotent follows by pulling back from $\K(\QQ, d)$ (e.g.\ using \cite[Lemma~2.4.7]{I}).
  Furthermore note that \cref{con:Cinf_circle_functor} yields a commutative square
  \[
  \begin{tikzcd}
    S \rar{\iota_S} \dar[swap]{\phi} & S_\alpha \dar{\phi_\alpha} \\
    R \rar{\iota_R} & R_{\phi(\alpha)}
  \end{tikzcd}
  \]
  and that $\phi_\alpha$ is a quasi-isomorphism by \cref{lemma:Cinf_circle_functor_qiso}.
  This implies the claim.
\end{proof}

\begin{remark} \label{rem:Cinf_circle}
  We now sketch an alternative way of proving \cref{thm:Cinf_circle}, and of coming up with the formulas of \cref{con:Cinf_circle} without having to guess.
  A unital $\Cinf$-algebra structure on $S = \QQ \langle 1 \rangle \oplus \wt S$ corresponds to a differential on the quasi-free graded Lie algebra $L \defeq \LL(\shift[-1] \dual {\wt S})$ (see e.g.\ \cite[Proposition~1.41]{BS23}), and $\alpha \in S^d$ to a map $\beta \colon L \to \QQ \langle b \rangle$ from $L$ to the abelian dg Lie algebra $\QQ \langle b \rangle$ with $b$ in degree $d-1$.
  In the situation of \cref{thm:Cinf_circle}, it is a model for the map $\xi \colon X \to \K(\QQ, d)$, and the kernel of $\beta$ is a model for $X_\xi$ (as long as $\beta$ is non-trivial). 
  This kernel is quasi-free, and corresponds to a $\Cinf$-algebra structure on $\quot S \alpha \oplus \shift[d-1] \wt S$; however, this is not our preferred model.

  Instead, writing $V \defeq \shift[-1] \dual {\wt S}$, we consider the inclusion of dg Lie algebras $L \to \LL(V, u, b)$ with $\deg u \defeq d-2$, $\deg b \defeq d-1$, and differential determined by $d(b) \defeq u$ and $d(u) \defeq 0$.
  The map $\beta$ extends to a map $\beta' \colon \LL(V, u, b) \to \QQ \langle b \rangle$ by specifying $\beta'(u) \defeq 0$ and $\beta'(b) \defeq b$.
  It is still a model for $\xi$ since the inclusion $L \to \LL(V, u, b)$ is a quasi-isomorphism.
  Hence $X_\xi$ is modeled by the kernel of $\beta'$.
  This kernel is generated by the elements $v - \beta(v) b$ for $v \in V$, the elements $[v, b]$ for $v \in V$, and the elements $u$, $[u, b]$, and $[b, b]$.
  We have $d([b, b]) = 2[u, b]$, and hence these two elements span an acyclic ideal.
  Quotienting it out, and rewriting the differential in terms of the remaining elements, we obtain a unital $\Cinf$-algebra structure on $S \tensor \QQ \langle 1, a \rangle$ (with $[v, b]$ corresponding to $\dual v \tensor a$, and $u$ corresponding to $1 \tensor a$).
  In the case that $\alpha = 0$, this recovers a special case of a result of Lupton--Smith \cite[Theorem~2.1]{LS07}.
\end{remark}

We conclude this appendix with the following alternative proof of (a reformulation of) \cref{prop:transfer_circle_composite}.

\begin{proposition} \label{prop:transfer_circle_composite_alt}
  In the situation of \cref{prop:transfer_circle_composite}, we have that the Becker--Gottlieb transfer $(f \after g)^! \colon \SS[B] \to \SS[X_\xi]$ is modeled by the map of cochain complexes $S_\alpha \to \k$ given by
  \[ x \tensor 1  \longmapsto  0  \displayand  x \tensor a  \longmapsto  - \tr^{\shift S}_\k \bigl( \mu^S_3(\shift x, \blank, \shift \alpha) \bigr) \]
  for $x \in S$ (using notation as in \cref{con:Cinf_circle}).
\end{proposition}

\begin{proof}
  By \cref{thm:Cinf_circle}, the $\Cinf$-algebra $R_\alpha$ over $\k$ is a model for $X_\xi$ over $B$.
  Note that the fiber of $f \after g$ is simply connected since the restriction of $\xi$ in $\Coho 2(F; \ZZ)$ is non-trivial.
  Thus, by \cite[\IIBGtransfer]{II}, the Becker--Gottlieb transfer $(f \after g)^!$ is modeled by the map of cochain complexes $S_\alpha \to \k$ given by
  \[ x \tensor b  \longmapsto  - \tr^{\shift S_\alpha}_\k \bigl( \mu^{S_\alpha}_2(\shift (x \tensor b), \blank) \bigr) \]
  for $b \in \{ 1, a \}$.
  Since
  \begin{align*}
    \mu^{S_\alpha}_2 \bigl( \shift (x \tensor 1), \shift (y \tensor 1) \bigr) &= \mu^S_2(\shift x, \shift y) \tensor 1 \\
    \mu^{S_\alpha}_2 \bigl( \shift (x \tensor 1), \shift (y \tensor a) \bigr) &= \mu^S_2(\shift x, \shift y, \shift \alpha) \tensor 1 + \mu^S_2(\shift x, \shift y) \tensor a
  \end{align*}
  we have
  \[ \tr^{\shift S_\alpha}_\k \bigl( \mu^{S_\alpha}_2(\shift (x \tensor 1), \blank) \bigr)  =  \tr^{\shift S}_\k \bigl( \mu^S_2(\shift x, \blank) \bigr) - \tr^{\shift S}_\k \bigl( \mu^S_2(\shift x, \blank) \bigr) = 0 \]
  (since $a$ has odd degree).
  Similarly, we have
  \begin{align*}
    \mu^{S_\alpha}_2 \bigl( \shift (x \tensor a), \shift (y \tensor 1) \bigr) &= \mu^S_3(\shift x, \shift \alpha, \shift y) \tensor 1 + (-1)^{\deg {\shift y}} \mu^S_3(\shift x, \shift y, \shift \alpha) \tensor 1 + (-1)^{\deg {\shift y}} \mu^S_2(\shift x, \shift y) \tensor a \\
    \mu^{S_\alpha}_2 \bigl( \shift (x \tensor a), \shift (y \tensor a) \bigr) &= \mu^S_4(\shift x, \shift \alpha, \shift y, \shift \alpha) \tensor 1 + \mu^S_3(\shift x, \shift \alpha, \shift y) \tensor a
  \end{align*}
  and thus
  \[ \tr^{\shift S_\alpha}_\k \bigl( \mu^{S_\alpha}_2(\shift (x \tensor a), \blank) \bigr)  =  \tr^{\shift S}_\k \bigl( \mu^S_3(\shift x, \blank, \shift \alpha) \bigr) \]
  which completes the proof.
\end{proof}

%% file: Tseries.bib
@online{I,
  author = {Naef, Florian and Stoll, Robin},
  title = {A rational model for the fiberwise THH transfer I: Sullivan algebras},
  note = {Preprint},
  year = {2026},
  eprinttype = {arxiv},
  eprint = {2604.02516v2}
}

@online{II,
  author = {Naef, Florian and Stoll, Robin},
  title = {A rational model for the fiberwise THH transfer II: $A_\infty$-algebras},
  note = {Preprint},
  year = {2026},
  eprinttype = {arxiv},
  eprint = {2604.24709v1}
}


%% file: references.bib
@article{AK,
  author = {Armenta, Marco Antonio  and  Keller, Bernhard},
  title = {Derived invariance of the Tamarkin--Tsygan calculus of an algebra},
  volume = {357},
  DOI = {10.1016/j.crma.2019.01.007},
  number = {3},
  journal = {Comptes Rendus. Mathématique},
  year = {2019},
  pages = {236--240}
}

@article{Bar17,
  title = {Spectral Mackey functors and equivariant algebraic K-theory (I)},
  DOI = {10.1016/j.aim.2016.08.043},
  journal = {Advances in Mathematics},
  author = {Barwick, Clark},
  year = {2017},
  pages = {646--727}
}

@article{BG76,
  author = {Becker, J. C. and Gottlieb, D. H.},
  title = {Transfer maps for fibrations and duality},
  journal = {Compositio Mathematica},
  pages = {107--133},
  volume = {33},
  number = {2},
  year = {1976},
  url = {https://www.numdam.org/item/CM_1976__33_2_107_0/}
}

@article{BG91,
  author = {Becker, James C.  and  Gottlieb, Daniel Henry},
  title = {Vector fields and transfers},
  journal = {Manuscripta Mathematica},
  volume = {72},
  number = {2},
  year = {1991},
  pages = {111--130},
  DOI = {10.1007/bf02568269}
}

@unpublished{Ber,
  author = {Berglund, Alexander},
  title = {Poincaré duality fibrations and graph complexes},
  note = {In progress}
}

@incollection {BS23,
  author = "Berglund, Alexander and Stoll, Robin",
  title = "Higher Structures in Rational Homotopy Theory",
  editor = "Vallette, Bruno",
  booktitle = "Higher Structures and Operadic Calculus",
  publisher = "Birkhäuser",
  series = "Advanced Courses in Mathematics -- CRM Barcelona",
  year = "2025",
  pages = "1--58",
  DOI = "10.1007/978-3-031-77779-0_1"
}

@article{BG,
  author = {Bousfield, A. K. and Gugenheim, V. K. A. M.},
  title = {On PL De Rham theory and rational homotopy type},
  volume = {8},
  number = {179},
  journal = {Memoirs of the American Mathematical Society},
  year = {1976}
}

@article{CCRY,
  author = {Carmeli, Shachar and Cnossen, Bastiaan and Ramzi, Maxime and Yanovski, Lior},
  title = {Characters and transfer maps via categorified traces},
  journal = {Forum of Mathematics Sigma},
  volume = {13},
  eid = {e93},
  DOI = {10.1017/fms.2025.23},
  year = {2025}
}

@article{Cla,
  author = {Clapp, Mónica},
  title = {Duality and transfer for parametrized spectra},
  volume = {37},
  DOI = {10.1007/bf01234383},
  number = {1},
  journal = {Archiv der Mathematik},
  year = {1981},
  pages = {462--472}
}

@article{DWW,
  author = {Dwyer, W. and Weiss, M. and Williams, B.},
  title = {A parametrized index theorem for the algebraic $K$-theory Euler class},
  volume = {190},
  DOI = {10.1007/bf02393236},
  number = {1},
  journal = {Acta Mathematica},
  year = {2003},
  pages = {1--104}
}

@book{FHT,
  author = {Félix, Yves and Halperin, Stephen and Thomas, Jean-Claude},
  title = {Rational Homotopy Theory},
  series = {Graduate Texts in Mathematics},
  number = {205},
  year = {2012},
  publisher = {Springer},
  doi = {10.1007/978-1-4613-0105-9}
}

@online{Hau13,
  author = {Haugseng, Rune},
  title = {The Becker-Gottlieb transfer is functorial},
  year = {2013},
  eprinttype = {arxiv},
  eprint = {1310.6321},
  note = {Retracted}
}

@article{Kad,
  author = {Kadeishvili, T.},
  title = {Structure of the $A(\infty)$-algebra and the Hochschild and Harrison cohomologies},
  volume = {91},
  journal = {Proceedings of Tbilisi A. Razmadze Mathematical Institute},
  year = {1988},
  pages = {19--27},
  language = {Russian with an English summary}
}

@inproceedings {Kel,
  AUTHOR = {Keller, Bernhard},
  TITLE = {A-infinity algebras, modules and functor categories},
  BOOKTITLE = {Trends in representation theory of algebras and related topics},
  EDITOR = {de la Peña, José A. and Bautista, Raymundo},
  SERIES = {Contemporary Mathematics},
  NUMBER = {406},
  PAGES = {67--93},
  PUBLISHER = {American Mathematical Society},
  YEAR = {2006},
  DOI = {10.1090/conm/406/07654},
}

@article{Kel21,
  author = {Keller, Bernhard},
  title = {Hochschild (Co)homology and Derived Categories},
  volume = {47},
  DOI = {10.1007/s41980-021-00556-0},
  number = {1 supplement},
  journal = {Bulletin of the Iranian Mathematical Society},
  year = {2021},
  pages = {57--83}
}

@article{KM,
  author = {Klein, John R. and Malkiewich, Cary},
  title = {Corrigendum to “The transfer is functorial”},
  volume = {410},
  DOI = {10.1016/j.aim.2022.108752},
  journal = {Advances in Mathematics},
  year = {2022},
  eid = {108752}
}

@article{KMR,
  author = {Klein, John and Malkiewich, Cary and Ramzi, Maxime},
  title = {On the multiplicativity of the Euler characteristic},
  journal = {Proceedings of the American Mathematical Society},
  volume = {151},
  number = {11},
  pages = {4997--5006},
  year = {2023},
  DOI = {10.1090/proc/16498}
}

@online{LMN,
  author = {Lehner, Georg and Malkiewich, Cary and Naef, Florian},
  title = {The Becker--Gottlieb transfer is not functorial: An example based on units of group rings},
  year = {2026},
  note = {Preprint},
  eprinttype = {arxiv},
  eprint = {2610.18720v1}
}

@article{LM,
  author = {Lind, John A. and Malkiewich, Cary},
  title = {The transfer map of free loop spaces},
  volume = {371},
  DOI = {10.1090/tran/7497},
  number = {4},
  journal = {Transactions of the American Mathematical Society},
  year = {2018},
  pages = {2503--2552}
}

@article{LS24,
  author = {Lindell, Erik and Saleh, Bashar},
  title = {Representation stability for homotopy automorphisms},
  volume = {24},
  DOI = {10.2140/agt.2024.24.2673},
  number = {5},
  journal = {Algebraic \& Geometric Topology},
  year = {2024},
  pages = {2673--2705}
}

@article{LS07,
  author = {Lupton, Gregory and Smith, Samuel Bruce},
  title = {Rationalized evaluation subgroups of a map II: Quillen models and adjoint maps},
  volume = {209},
  DOI = {10.1016/j.jpaa.2006.05.019},
  number = {1},
  journal = {Journal of Pure and Applied Algebra},
  year = {2007},
  pages = {173--188}
}

@article{Pri,
  author = {Prigge, Nils},
  title = {Tautological rings of fibrations},
  volume = {26},
  DOI = {10.2140/agt.2026.26.1565},
  number = {4},
  journal = {Algebraic \& Geometric Topology},
  year = {2026},
  pages = {1565--1584}
}

@article{Smi,
  author = {Smirnov, V. A.},
  title = {Homology of fibre spaces},
  volume = {35},
  DOI = {10.1070/rm1980v035n03abeh001856},
  number = {3},
  journal = {Russian Mathematical Surveys},
  year = {1980},
  pages = {294--298}
}

@article{Sta,
  author = {Stasheff, James Dillon},
  title = {Homotopy Associativity of H-Spaces. II},
  volume = {108},
  number = {2},
  journal = {Transactions of the American Mathematical Society},
  year = {1963},
  pages = {293--312},
  DOI = {10.2307/1993609}
}

@article{Sul,
  author = {Dennis Sullivan},
  title = {Infinitesimal computations in topology},
  journal = {Publications ma\-thé\-ma\-tiques de l'IHÉS},
  year = {1977},
  volume = {47},
  number = {1},
  pages = {269--331},
  doi = {10.1007/bf02684341}
}

@book{Tan,
  author = {Daniel Tanré},
  title = {Homotopie rationnelle},
  subtitle = {Modèles de Chen, Quillen, Sullivan},
  year = {1983},
  publisher = {Springer},
  series = {Lecture Notes in Mathematics},
  number = {1025},
  doi = {10.1007/bfb0071482}
}
